\documentclass[11pt]{amsart}

\usepackage[T1]{fontenc}
\usepackage[utf8]{inputenc}
\usepackage{textcomp}
\usepackage{lmodern}
\usepackage{microtype}

\usepackage{amssymb}
\usepackage{amsthm}
\usepackage{amscd}

\usepackage{hyperref}

\usepackage{mathscinet}
\usepackage[giveninits=true,doi=false,url=false,sortcites,backend=biber]{biblatex}
\newtheorem{theorem}{Theorem}[section]

\newtheorem{lemma}[theorem]{Lemma}

\theoremstyle{definition}
\newtheorem{definition}[theorem]{Definition}

\newtheorem{remark}[theorem]{Remark}

\allowdisplaybreaks[2]

\begin{document}

\title{More or Less Uniform Convergence}
\author{Henry Towsner}
\date{\today}
\thanks{Partially supported by NSF grant DMS-1600263}
\address {Department of Mathematics, University of Pennsylvania, 209 South 33rd Street, Philadelphia, PA 19104-6395, USA}
\email{htowsner@math.upenn.edu}
\urladdr{\url{http://www.math.upenn.edu/~htowsner}}

\begin{abstract}
Uniform metastable convergence is a weak form of uniform convergence for a family of sequences.  In this paper we explore the way that metastable convergence stratifies into a family of notions indexed by countable ordinals.

We give two versions of this stratified family; loosely speaking, they correspond to the model theoretic and proof theoretic perspectives.  For the model theoretic version, which we call abstract $\alpha$-uniform convergence, we show that uniform metastable convergence is equivalent to abstract $\alpha$-uniform convergence for some $\alpha$, and that abstract $\omega$-uniform convergence is equivalent to uniformly bounded oscillation of the family of sequences.

The proof theoretic version, which we call concrete $\alpha$-uniform convergence, is less canonical (it depends on a choice of ordinal notation), but appears naturally when ``proof mining'' convergence proofs to obtain quantitative bounds.

We show that these hierarchies are strict by exhibiting a family of which is concretely $\alpha+1$-uniformly convergent but not abstractly $\alpha$-uniformly convergent for each $\alpha<\omega_1$.
\end{abstract}

\maketitle

\section{Introduction}

Our goal in this paper is to collect up the relationships between some notions of uniform convergence.  The notion of metastable convergence introduced in \cite{tao08Norm,avigad:MR2550151} (and earlier studied in \cite{MR2144170,MR2130066}) can be seen as a family of notions indexed by ordinals, with full metastable convergence corresponding to the $\omega_1$ level, and the notion of bounded oscillation corresponding to the $\omega$ level, with other notions in between.  While this idea has appeared implicitly in the literature \cite{gaspar,MR3643744}, the details have not been made explicit.

We define the general family of notions of abstract $\alpha$-uniform convergence and show that uniform metastable convergence is equivalent to abstract $\alpha$-uniform convergence for some $\alpha<\omega_1$.  We introduce another, slightly stronger notion, concrete $\alpha$-uniform convergence, which has the benefit of being more explicit but the disadvantage of depending on explicit representations of ordinals (specifically, fixed sequences $\alpha_n$ so that $\alpha=\sup_n (\alpha_n+1)$ for each $\alpha$).  Finally, we introduce families of sequences $\mathcal{S}_\alpha$ (the sequences which ``change value $\alpha$ times'') and show that each $\mathcal{S}_\alpha$ is concretely $\alpha+1$-uniformly convergent but not abstractly $\alpha$-uniformly convergent.

\section{Uniform Metastable Convergence}

Throughout this paper, we will focus on $\{0,1\}$-valued sequences, and we write $\bar a$ for the sequence $(a_n)_{n\in\mathbb{N}}$.  The ideas generalize to sequences valued in any complete metric space, and we will occasionally discuss this generalization in remarks.

We are interested in sets $\mathcal{S}$ of sequences such that every $\bar a\in\mathcal{S}$ converges.  In particular, we are interested in questions about the uniformity of this convergence.  The classic notion of uniform convergence---that there is some fixed $m$ so that, for every $\bar a\in\mathcal{S}$, if $m,m'\geq n$ then $a_m=a_{m'}$---is quite strong.  (Indeed, in our restricted setting of $\{0,1\}$-valued sequences, it is easy to see that any uniformly convergent set of sequences is finite.)

The following, weaker notion, is in some sense the weakest reasonable notion of uniformity.
\begin{definition}
    Let $\mathcal{S}$ be a set of sequences.  We say $\mathcal{S}$ \emph{converges uniformly metastably} if for every $F:\mathbb{N}\rightarrow\mathbb{N}$ such that $n<F(n)$ and $F(n)\leq F(n+1)$ for all $n$, there exists an $M_F$ so that, for every $\bar a\in\mathcal{S}$, there is an $m\leq M_F$ and a $c\in\{0,1\}$ so that for all $n\in[m,F(m)]$, $a_n=c$.
\end{definition}
This notion has also been called \emph{local stability} \cite{avigad:MR2550151}: it says that we can find very long intervals on which the sequence $\bar a$ has stabilized, where the length of the interval can even depend on how large the starting point of the interval is (that is, the interval has the form $[m,F(m)]$, where $F$ could grow very quickly as a function of $m$).

A variety of results \cite{tao08Norm, avigad:MR2550151,MR3278188,MR2563097,2016arXiv161005397C,MR3141811,iovino_duenez} have shown that, under some assumptions on $\mathcal{S}$, if every sequence in $\mathcal{S}$ converges then actually $\mathcal{S}$ must converge uniformly metastability.  These hypotheses are usually given in terms of logic, but in this simple setting a direct formulation is possible.

\begin{definition}
  Suppose that, for each $i$, $\bar a^i=(a_n^i)_{n\in\mathbb{N}}$ is a sequence.   A \emph{limit sequence} is a sequence $\bar b=(b_n)_{n\in\mathbb{N}}$ such that there is an infinite set $S$ so that, for each $n$, $\{i\in S\mid a^i_n\neq b_n\}$ is finite.
\end{definition}

For example, suppose that $a_n^i=\left\{\begin{array}{ll}1&\text{if }i<n\\0&\text{otherwise}\end{array}\right.$.  Then the only limit sequence is the sequence which is constantly equal to $0$.

More generally, if $\lim_{i\rightarrow\infty}a^i_n$ converges to a value $c$ then, in any limit sequence $(b_n)$, $b_n=c$.  When $\lim_{i\rightarrow\infty}a^i_n$ fails to converge, there must be multiple limits, including at least one where $b_n=0$ and one where $b_n=1$; the requirement that there is a single set $S$ is a coherence condition.

For example, if $a_n^i=\left\{\begin{array}{ll}1&\text{if }2^n\mid i\\0&\text{otherwise}\end{array}\right.$ then the possible limits are the sequence that is all $1$'s, or any sequence consisting of a finite, positive number of $1$'s followed by $0$'s: the choice where there are $k\geq 1$ $1$'s followed by $0$'s corresponds to take $S$ to be, for example, the numbers divisible by $2^{k-1}$ but not $2^k$; the limit which is constantly $1$ corresponds to taking $S$ to be, for instance, $\{1,2,4,8,16,\ldots\}$.

\begin{lemma}
  For any collection of sequences $\bar a^i$, there exists a limit sequence.
\end{lemma}

\begin{remark}
  In the more general setting \cite{MR3141811}, the notion of a limit sequence is replaced by an \emph{ultraproduct}.  This not only allows consideration of sequences valued in arbitrary metric spaces, it includes the case where different sequences come from different metric spaces.
\end{remark}

\begin{definition}
  We say a sequence $\bar a=(a_n)_{n\in\mathbb{N}}$ \emph{converges} if there is some $m$ and some $c\in\{0,1\}$ so that for all $n\geq m$, $a_n=c$.

  Let $\mathcal{S}$ be a set of sequences.  We say $\mathcal{S}$ \emph{converges uniformly metastably} if for every $F:\mathbb{N}\rightarrow\mathbb{N}$, there exists an $M_F$ so that, for every $\bar a\in\mathcal{S}$, there is an $m\leq M_F$ and a $c\in\{0,1\}$ so that for all $n\in[m,F(m)]$, $a_n=c$.
\end{definition}

\begin{theorem}
  Suppose $\mathcal{S}$ is a set of sequences such that:
  \begin{itemize}
  \item every sequence in $\mathcal{S}$ converges, and
  \item whenever $\{\bar a^i\}_{i\in\mathbb{N}}\subseteq\mathcal{S}$, every limit sequence of $\bar a^i$ is also in $\mathcal{S}$.
  \end{itemize}
Then $\mathcal{S}$ converges uniformly metastably.
\end{theorem}
\begin{proof}
  Suppose $\mathcal{S}$ does not converge uniformly metastably; then there is an $F:\mathbb{N}\rightarrow\mathbb{N}$ so that, for each $i$, there is a $\bar a^i\in\mathcal{S}$ such that, for every $m\leq i$, there are $n_0,n_1\in[m,F(m)]$ so that $a_{n_0}=0$ and $a_{n_1}=1$.

  Let $\bar b$ be some limit of the $\bar a^i$.  Then $\bar b\in\mathcal{S}$, so $\bar b$ converges.  So there is some $m$ and some $c$ so that, for all $n\geq m$, $b_n=c$.

  Choose an infinite set $S$ so that, for all $n$, $\{i\in S\mid a^i_n\neq b_n\}$ is finite.  In particular, the set of $i\in S$ such that, for all $n\in[m,F(m)]$, $a^i_n=b_n$ must be infinite, so we can find some $i\geq m$ in $S$.  Then for every $n\in[m,F(m)]$, $a^i_n=b_n=c$.  But this contradicts the choice of $\bar a^i$.
\end{proof}

We wish to associate sets of convergent sequences $\mathcal{S}$ to ordinals.



\begin{definition}
Let $\mathcal{S}$ be a set of sequences.  We define $T_{\mathcal{S}}$ to be the tree of finite increasing sequences $0<r_1<\cdots<r_M$ such that, taking $r_0=0$, there is some $\bar a\in\mathcal{S}$ so that, for every $i<M$, there are $n_0,n_1\in[r_i,r_{i+1}]$ with $a_{n_0}=0$ and $a_{n_1}=1$.

When $\alpha<\omega_1$ is an ordinal, we say that $\mathcal{S}$ converges abstractly $\alpha$-uniformly if $T_{\mathcal{S}}$ has height strictly less than $\alpha$.
\end{definition}

\begin{remark}
  Note that, in our setting, $\mathcal{S}$ converges abstractly $\omega$-uniformly iff $\mathcal{S}$ converges abstractly $n$-uniformly for some finite $n$ (and similarly for other limit ordinals).

  When considering more general sequences, one would want consider countably many trees corresponding to fluctuations of size approaching $0$: we could consider the tree $T_{\mathcal{S},k}$ of sequences $0=r_0<r_1<\cdots<r_M$ such that there is some $\bar a\in\mathcal{S}$ so that, for every $i<M$, there are $n_0,n_1\in[r_i,r_{i+1}]$ with $|a_{n_0}-a_{n_1}|>1/k$.  Then we would say $\mathcal{S}$ converges abstractly $\alpha$-uniformly if for each $k$, $T_{\mathcal{S},k}$ has height strictly less than $\alpha$.
\end{remark}

To connect this with metastability, we need to relate these sequences $(r_i)_{1\leq i\leq M}$ to functions $F:\mathbb{N}\rightarrow\mathbb{N}$.

\begin{definition}
  Given a strictly increasing sequence $\bar r$ with $0=r_0$, we define $F_{\bar r}(n)$ by taking $i$ least so $n\leq r_i$ and setting $F_{\bar r}(n)=r_{i+1}$.

For any function $F:\mathbb{N}\rightarrow\mathbb{N}$ such that $n<F(n)$ and $F(n)\leq F(n+1)$ for all $n$, we define a function $\hat F=F_{(F^{i}(0))_{i\in\mathbb{N}}}$.
\end{definition}

\begin{lemma}
  For any $F:\mathbb{N}\rightarrow\mathbb{N}$ such that $n<F(n)$ and $F(n)\leq F(n+1)$ for all $n$, $F(n)\leq \hat F(n)$ for all $n$.
\end{lemma}
\begin{proof}
  Let $r_i=F^i(0)$ for all $i$, so $\hat F=F_{\bar r}$.

Let $i$ be least so that $n\leq r_i=F^i(0)$.  Then $\hat F(n)=F^{i+1}(0)=F(F^i(0))\geq F(n)$.

\end{proof}

Note that if $M_{\hat F}$ witnesses metastability for $\hat F$ then it also witnesses metastability for $F$; in particular, this means that if we can show metastability for all functions of the form $F_{\bar r}$, we have shown metastability.

\begin{theorem}\label{thm:main}
  $\mathcal{S}$ converges uniformly metastably iff $\mathcal{S}$ converges abstractly $\alpha$-uniformly for some $\alpha<\omega_1$.
\end{theorem}
\begin{proof}
  Suppose $\mathcal{S}$ does not converge abstractly $\alpha$-uniformly for any $\alpha<\omega_1$---that is, suppose the tree $T_{\mathcal{S}}$ is ill-founded.  Then it has an infinite path $\bar r$, and the function $F_{\bar r}$ witnesses a failure of uniformly metastable convergence: by the definition of $T_{\mathcal{S}}$, for every $M$ there is a $\bar a\in\mathcal{S}$ so that, for each $m\leq M$, letting $i$ be least so $m\leq r_i$, there are $n_0,n_1\in [r_i,r_{i+1}]\subseteq [m,F_{\bar r}(m)]$ with $a_{n_0}=0$ and $a_{n_1}=1$.  Therefore the uniform bound $M_F$ cannot exist.

  Conversely, suppose $F$ witnesses a failure of uniformly metastable convergence, so also $\hat F$ witnesses a failure of uniformly metastable convergence.  Then for each $K$, there is a sequence $\bar a$ so that for each $k< K$ there are $n_0,n_1\in[k,\hat F(k)]$ so that $a_{n_0}=0$ and $a_{n_1}=1$.  In particular, taking $K=F^M(0)$, we may find $\bar a$ so that for each $i<M$ there are such $n_0,n_1\in[F^i(0),F^{i+1}(0)]$.  Therefore the sequence $(F^i(0))_{i\in\mathbb{N}}$ is an infinite branch through $T_{\mathcal{S}}$.
\end{proof}

We mention one other notion of uniform convergence, which has been particularly studied \cite{1705.10355,MR3345161,MR3111912,bishop:MR0228655}: bounds on jumps (which, in this context, are essentially the same as bounds on ``oscillations'' or ``upcrossings'' as they are sometimes known in the literature).

\begin{definition}
  $\mathcal{S}$ has \emph{uniformly bounded jumps} if there is a $k$ so that whenever $\bar a\in\mathcal{S}$ and $n_0<\cdots<n_{k}$ is a sequence, there is an $i<k$ so that $a_{n_i}=a_{n_{i+1}}$.
\end{definition}

\begin{theorem}
  $\mathcal{S}$ has uniformly bounded jumps iff $\mathcal{S}$ converges abstractly $\omega$-uniformly.
\end{theorem}
\begin{proof}
  Suppose $\mathcal{S}$ does not converge abstractly $\omega$-uniformly, so for each $k$ there is a sequence $0<r^k_0<\cdots<r^k_k$ and a $\bar a\in\mathcal{S}$ so that for each $i<k$ there are $n^i_0,n^i_1\in[r^k_i,r^k_{i+1}]$ with $a_{n^i_0}=0$ and $a_{n^i_1}=1$.  Then the sequence $n^0_0<n^1_1<n^0_2<n^1_3<\cdots$ shows that $\mathcal{S}$ does not have uniformly bounded jumps with bound $k$.  Since this holds for any $k$, $\mathcal{S}$ does not have uniformly bounded jumps.

  Conversely, suppose $\mathcal{S}$ does not have uniformly bounded jumps, so for each $k$ there is an $\bar a$ and a sequence $n_0<\cdots<n_k$ so that $a_{n_i}\neq a_{n_{i+1}}$ for each $i<k$.  Then the sequence $n_1,n_3,n_5,\ldots$ belongs to $T_{\mathcal{S}}$.  Since $T_{\mathcal{S}}$ contains arbitrarily long finite sequences, $T_{\mathcal{S}}$ has height at least $\omega$.
\end{proof}

\section{Ordinal Iterations}

\subsection{Concrete $\alpha$-Uniform Convergence}

Most naturally occuring examples of abstract $\alpha$-uniform convergence satisfy a stronger property.  This stronger property is not quite canonical---we need to fix a family of ordinal notations.

\begin{definition}
  A \emph{fundamental sequence}\footnote{Various definitions of this notion which are not exactly equivalent are found in the literature, but the differences are generally minor.} for a countable ordinal $\alpha>0$ is a sequence of ordinals $\alpha[n]$ for $n\in\mathbb{N}$ such that:
  \begin{itemize}
  \item $\alpha[n]<\alpha$,
  \item $\alpha[n]\leq\alpha[n+1]$, and
  \item for every $\beta<\alpha$, there is an $n$ with $\beta\leq\alpha[n]$.
  \end{itemize}

For convenience, we define $0[n]=0$.
\end{definition}
When $\alpha$ is a successor---$\alpha=\gamma+1$---these conditions imply that $\alpha[n]=\gamma$ for all but finitely many $n$.  When $\alpha$ is a limit ordinal, these conditions imply that $\lim_{n\rightarrow\infty}\alpha[n]=\alpha$.

For small ordinals, there are conventional choices of fundamental sequences, like $\omega[n]=n$, $\omega^2[n]=\omega\cdot n$, $\epsilon_0[n]=\omega_n$ (where $\omega_0=0$ and $\omega_{n+1}=\omega^{\omega_n}$), and so on, arising out of ordinal notation schemes.

For the remainder of the paper, assume we have fixed, for every countable ordinal $\alpha$ we consider, some fundamental sequence.  For convenience, we assume that $(\gamma+1)[n]=\gamma$ for all successor ordinals.

\begin{definition}
  Let $F:\mathbb{N}\rightarrow\mathbb{N}$ be a function with $F(n)> n$ and $F(n)\leq F(n+1)$ for all $n$.  We define the $\alpha$-iteration of $F$ by:
  \begin{itemize}
  \item $F^0(n)=n$,
  \item when $\alpha>0$, $F^\alpha(n)=F^{\alpha[F(n)]}(F(n))$.
  \end{itemize}
\end{definition}
Then $F^1$ is just $F$, $F^k$ is the usual $k$-fold iteration of $F$, $F^\omega(n)=F^{F(0)+1}(0)$ (assuming the conventional fundamental sequence $\omega[n]=n$ for $\omega$), and so on.  Note that the definition of this iteration does depend on the choice of fundamental sequences. 

Note that these functions are not quite increasing in the ordinal: if $\alpha<\beta$ then we have $F^\alpha(n)\leq F^\beta(n)$ for sufficiently large $n$, but not necessarily when $n$ is small.  (For instance, compare $F^{1000}(3)$ to $F^\omega(3)$ for $F$ not growing too quickly.)

When calculating $F^\alpha(n)$, there is a canonical sequence of values and ordinals associated with its computation, given by
\begin{itemize}
\item $r_0=n$, $\beta_0=\alpha$,
\item $r_{i+1}=F(r_i)$, $\beta_{i+1}=\beta_i[r_{i+1}]$.
\end{itemize}
Since the sequence $\beta_0,\beta_1,\ldots$ is strictly decreasing, it terminates at some value $k$ with $\beta_k=0$, and we have
\[F^\alpha(n)=F^{\beta_0}(r_0)=F^{\beta_1}(r_1)=\cdots=F^{\beta_k}(r_k)=r_k.\]

\begin{definition}
  We say $\mathcal{S}$ converges concretely $\alpha$-uniformly if there is a $\beta<\alpha$ so that, for every $F:\mathbb{N}\rightarrow\mathbb{N}$ such that $F(n)> n$ for all $n$, for each $\bar a\in \mathcal{S}$, there is an $m$ with $F(m)\leq F^\beta(0)$ and a $c$ so that, for all $n\in[m,F(m)]$, $a_n=c$.
\end{definition}

\begin{remark}
  Again, in our restricted setting this is only interesting at successor ordinals where concrete $\alpha+1$-uniform convergence means $F^\alpha(0)$ always suffices as a bound.

  With more general sequences, we would say that for each $k$ there is a $\beta<\alpha$ so that, for every $F:\mathbb{N}\rightarrow\mathbb{N}$ such that $F(n)> n$ for all $n$, for each $\bar a\in \mathcal{S}$, there is an $m$ with $F(m)\leq F^\beta(0)$ and a $c$ so that, for all $n,n'\in[m,F(m)]$, $|a_n-a_{n'}|<1/k$.
\end{remark}

\begin{lemma}
  If $\mathcal{S}$ converges concretely $\alpha$-uniformly then $\mathcal{S}$ converges abstractly $\alpha$-uniformly.
\end{lemma}
\begin{proof}
Suppose $\mathcal{S}$ fails to converge abstractly $\alpha$-uniformly, so $T_{\mathcal{S}}$ has height $\geq\alpha$ (possibly ill-founded).  For each $\beta<\alpha$, we will construct a function $F$ and find a $\bar a\in\mathcal{S}$ witnessing the failure of strong uniform convergence.

Fix some $\beta<\alpha$.  By induction on $n$, we choose a decreasing sequence of ordinals $\beta_n\leq\beta$ and sequence $(r_i)_{i\leq n}$.  We begin with $r_0=0$ and $\beta_0=\beta$.  Given $\beta_n$, we take $r_{n+1}$ so that the set of sequences in $T_{\mathcal{S}}$ extending $(r_i)_{1\leq i\leq n+1}$ has height $\geq\beta_n$.  We then set $\beta_{n+1}=\beta_n[r_{n+1}]$.  We continue until we reach some $k$ so that $\beta_k=0$.

Chose $\bar a\in\mathcal{S}$ so that, for each $i<k$, there are $n_0,n_1\in[r_i,r_{i+1}]$ with $a_{n_0}=0$ and $a_{n_1}=1$.

We extend the sequence $(r_i)_{1\leq i\leq k}$ arbitrarily (say $r_i=r_k+i$ for $i>k$) and set $F=F_{\bar r}$.  Then the computation sequence for $F^\beta(0)$ is precisely
\[F^{\beta_0}(r_0)=F^{\beta_1}(r_1)=\cdots=F^{\beta_k}(r_k)=r_k.\]
Suppose $F(m)\leq F^\beta(0)=r_k$, so $m\leq r_{k-1}$.  Then, for some $i<k$, we have $[r_i,r_{i+1}]\subseteq [m,F(m)]$, and therefore there are $n_0,n_1\in[m,F(m)]$ with $a_{n_0}=0$ and $a_{n_1}=1$.

Since we can construct some such $F$ for any $\beta<\alpha$, $\mathcal{S}$ is not concretely $\alpha$-uniformly convergent.
\end{proof}

As a syntactic analog to Theorem \ref{thm:main}, we expect that if we prove that $\mathcal{S}$ converges uniformly metastably in some reasonable theory $T$ with proof-theoretic ordinal $\lambda$, then there should be some $\alpha<\lambda$ such that we can prove that $\mathcal{S}$ converges concretely $\alpha$-uniformly.  See \cite{1609.07509} for an explicit example of such an analysis in the context of differential algebra.

\subsection{A Proper Hierarchy}

We show that these notions form a proper hierarchy by constructing, for each $\alpha$, a family $\mathcal{S}_\alpha$ of sequences which are concretely $\alpha$-uniformly convergent but not abstractly $\alpha+1$-uniformly convergent.

\begin{definition}
We define $\alpha[c_1,\ldots,c_m]$ inductively by $\alpha[c_1,\ldots,c_{m+1}]=(\alpha[c_1,\ldots,c_m])[c_{m+1}]$.

  We define $\mathcal{S}_\alpha$ to consist of those sequences $(a_n)_{n\in\mathbb{N}}$ such that whenever $c_1,\ldots,c_k$ is a sequence such that:
  \begin{itemize}
  \item $c_1$ is least so that $a_{c_1}\neq a_0$,
  \item for each $i<k$, $c_{i+1}$ is the smallest value greater than $c_i$ so that $a_{c_{i+1}}\neq a_{c_i}$,
  \end{itemize}
then $\alpha[c_1,\ldots,c_k]\neq 0$.
\end{definition}
Roughly speaking, this measures how many times the sequence changes from being $0$'s to being $1$'s.  For example, $\mathcal{S}_k$ is the sequence which changes at most $k$ times.  $\mathcal{S}_\omega$ is the sequence where, taking $n$ to be the first place where the sequence changes, it changes at most $n$ additional times.  The condition is essentially that if we collect the ``runs'' of consecutive $1$'s or $0$'s, the starting points form at $\alpha$-large set in the sense of \cite{MR606791}.

Note that, for any $\bar a\in\mathcal{S}_\alpha$, the statement that $\bar a$ belongs to $\mathcal{S}_\alpha$ really concerns some maximal finite sequence $c_1,\ldots,c_k$.  However it is convenient to phrase the definition this way---where we also consider initial segments $c_1,\ldots,c_{k'}$ for some $k'<k$---because when $\bar a\not\in\mathcal{S}_\alpha$, the sequence may be infinite, but some finite initial segment is long enough to witness that $\alpha$ reduces to $0$.

\begin{lemma}
If $\{\bar a^i\}_{i\in\mathbb{N}}\subseteq\mathcal{S}_\alpha$ then every limit sequence of $\bar a^i$ is also in $\mathcal{S}_\alpha$.
\end{lemma}
\begin{proof}
  Let $\bar a^i$ be given and consider some limit $\bar b$. witnessed by an infinite set $S$ such that, for all $n$, $\{i\in S\mid a^i_n\neq b_n\}$ is finite.  Take any sequence $c_1,\ldots,c_k$ for $\bar b$ as in the definition of $\mathcal{S}_\alpha$.  Then we may find an $i$ so that, for all $n\leq c_k$, $a^i_n=b_n$.  Then, since $\bar a^i\in\mathcal{S}_\alpha$, we have $\alpha[c_1,\ldots,c_k]\neq 0$ as needed.
\end{proof}

In order to prove the results we need, we need an additional property on our fundamental sequences
\begin{definition}
  Suppose that, for all $\beta\in(0,\alpha]$, we have a fundamental sequence for $\beta$.  We say these sequences are \emph{monotone} if, for each $\beta\leq\alpha$ and any sequences $r_1,\ldots,r_k$ and $r'_1,\ldots,r'_k$ with $r_i\leq r'_i$ for all $i$, $\beta[r_1,\ldots,r_k]\leq\beta[r'_1,\ldots,r'_k]$.
\end{definition}
The usual fundamental sequences on small ordinals are monotone.  With a non-monotone fundamental sequence, we could have $\omega^2[1]=10000$ while $\omega^2[2]=\omega$, and then we would have $\omega^2[1,1]=9999$ while $\omega^2[2,1]=1$; this is the sort of anomaly we need to avoid.

\begin{lemma}
  $\mathcal{S}_\alpha$ is concretely $\alpha+1$-uniformly convergent.
\end{lemma}
\begin{proof}
Let $F$ be given and $\bar a\in\mathcal{S}_\alpha$ and suppose towards a contradiction that for each $m\leq F^\alpha(0)$ there are $n_{m_0},n_{m_1}\in[m,F(m)]$ with $a_{n_{m_0}}=0$ and $a_{n_{m_1}}=1$.

Take the computation sequence for $F^\alpha(0)$, where $r_0=0$, $\beta_0=\alpha$, $r_{i+1}=F(r_i)$, and $\beta_{i+1}=\beta_i[r_{i+1}]$.  Let $k$ be least so $\beta_k=0$, so $F^\alpha(0)=r_k$.

Each interval $[r_i,r_{i+1}]$ (including $[0,r_1]$) must contain the start of at least one new run, so taking the sequence $c_1,\ldots,c_k$ corresponding to $\alpha$, we have $c_i\leq r_i$ for all $i$.  Since $0\neq\alpha[c_1,\ldots,c_k]\leq \alpha[r_1,\ldots,r_k]=\beta_k=0$, we have the desired contradiction.
\end{proof}

\begin{lemma}
  $\mathcal{S}_\alpha$ is not abstractly $\alpha$-uniformly convergent.
\end{lemma}
\begin{proof}
Given a sequence $(r_i)_{i\leq k}$ with $\alpha[r_1,\ldots,r_k]\neq 0$, we may naturally associate a sequence $\bar a\in\mathcal{S}_\alpha$ by taking, for each $n$, the unique $i$ such that $n\in[r_i,r_{i+1})$ (where $r_0=0$) and setting $a_n=i\mod 2$.  This witnesses that any such sequence belongs to $T_{\mathcal{S}_\alpha}$.

So it suffices to show that, for any $\alpha$, the tree of squences $(r_i)_{i\leq k}$ with $\alpha[r_1,\ldots,r_k]\neq 0$ has height $\geq\alpha$.  We need to show slightly more, namely that for any $\alpha$ and any $d$, the tree of sequences $(r_i)_{i\leq k}$ with $d<r_1$ and $\alpha[r_1,\ldots,r_k]\neq 0$ has height $\geq\alpha$.  We proceed by induction on $\alpha$; of course when $\alpha=0$, this tree is empty, and so has height $0$.

Suppose the claim holds for all $\beta<\alpha$ and let $d$ be given.  Fix any $r_1>d$.  The tree of sequences whose first element is $r_1$ and $\alpha[r_1,\ldots,r_k]$ is precisely the tree of sequences $(r_i)_{2\leq i\leq k}$ such that $r_1<r_2$ and $(\alpha[r_1])[r_2,\ldots,r_k]\neq 0$, which has height $\geq\alpha[r_1]$.  Since, $\sup_n \alpha[n]=\alpha$ (even after discarding those $n$ with $n\leq d$), the tree of these sequences has height $\geq\alpha$.
\end{proof}

\subsection{Distinguishing Abstract and Concrete $\alpha$-Uniformity}

Finally, we note that concrete $\alpha$-uniformity really is stronger than abstract $\alpha$-uniformity.  

To illustrate the gap at the $\omega+1$, take any sufficiently fast-growing function $f:\mathbb{N}\rightarrow\mathbb{N}$ (in fact, $f(n)=2n$ suffices).  Consider the family $\mathcal{S}_f$ consisting of:
\begin{itemize}
\item the sequence which is all $0$'s,
\item for infinitely many $n$, the sequence given by
\[a^n_i=\left\{\begin{array}{ll}
1&\text{if }i=n+2j\text{ for some }j<f(n)\\
0&\text{otherwise}
\end{array}\right..\]
\end{itemize}
That is, $\bar a^n$ is the sequence
\[00\cdots 00101010\cdots 010000\cdots\]
where the first $1$ occurs at the $n$-th position, there are $f(n)$ alternations, and then the sequence finishes with infinitely many $0$'s.  By choosing the set of $n$ sufficiently sparsely, we can ensure that if $\bar a,\bar b\in\mathcal{S}_f$ and $a_i=b_i=1$ then $\bar a=\bar b$.  This guarantees that any limit of $\mathcal{S}_f$ is also in $\mathcal{S}_f$.

For any $n$, consider the function $F$ such that $F(i)=n$ for $i<n$ and $F(i)=i+1$ for $i>n$.  Then
\[F^\omega(0)=F^{\omega[n]}(n)=2n.\]
But the sequence $\bar a^n$ has both a $0$ and a $1$ on every interval $[i,F(i)]$ with $F(i)\leq 2n$.

We could address this gap in an individual case by tweaking the definition of concrete $\alpha$-uniformity, either by using a different fundamental sequence for $\omega$, or by allowing concrete $\alpha$-uniformity to use a bound like $F^{\alpha+k}(d)$ for constants $k,d$ that depend on the family $\mathcal{S}_f$ (but not on the function $F$).  But, by choosing $f$ growing very fast, we can still find families $\mathcal{S}_f$ which outpace any fixed modification of this kind.

\printbibliography
\end{document}